\documentclass[11pt]{amsart}

\usepackage{amsmath,amssymb,latexsym,soul,cite,amsthm,xcolor,enumitem,graphicx,tikz, mathtools,microtype}
\usepackage{amsfonts,amsthm,amssymb,marvosym,esint,xfrac,bbm,amssymb}

\usepackage[normalem]{ulem}

\usepackage[colorlinks=true, urlcolor=blue, citecolor=red, linkcolor=blue]{hyperref}
\usepackage[english]{babel}
\usepackage[left=2.5cm,right=2.5cm,top=2.7cm,bottom=2.7cm]{geometry}
\usepackage{cancel}
\usepackage[dvipsnames]{xcolor}

\numberwithin{equation}{section}

\newtheorem{theorem}{Theorem}[section]
\theoremstyle{plain}
\newtheorem{lemma}[theorem]{Lemma}
\theoremstyle{plain}
\newtheorem{proposition}[theorem]{Proposition}
\theoremstyle{plain}

\newtheorem{definition}[theorem]{Definition}
\theoremstyle{definition}
\newtheorem{remark}[theorem]{Remark}

\def\Xint#1{\mathchoice
    {\XXint\displaystyle\textstyle{#1}}%
    {\XXint\textstyle\scriptstyle{#1}}%
    {\XXint\scriptstyle\scriptscriptstyle{#1}}%
    {\XXint\scriptscriptstyle\scriptscriptstyle{#1}}%
    \!\int}
\def\XXint#1#2#3{\setbox0=\hbox{$#1{#2#3}{\int}$}
    \vcenter{\hbox{$#2#3$}}\kern-0.5\wd0}
\def\bint{\Xint-}
\def\bint{\Xint-}
\def\dashint{\Xint{\raise4pt\hbox to7pt{\hrulefill}}}
\def\dashiint{\bint\kern-0.15cm\bint}

\title[Harnack-type inequalities for nonlocal parabolic equations...]{Harnack-type inequalities for nonlocal parabolic equations in the upper fast-diffusion range}
\author{M. Savchenko$^{1}$,  I. Skrypnik$^{1}$, Y. Yevgenieva$^{2,1}$}
\thanks{$^{1}$Institute of Applied Mathematics and Mechanics, National Academy of Sciences of Ukraine, 84116 Sloviansk, Ukraine}
\thanks{$^{2}$Max Planck Institute for Dynamics of Complex Technical Systems, 39106 Magdeburg, Germany}
\thanks{{\tt\small\ shan{\_}maria@ukr.net, ihor.skrypnik@gmail.com, yevgenieva@mpi-magdeburg.mpg.de}}

\begin{document}
\begin{abstract} 
We study non-negative weak solutions to a class of doubly nonlinear nonlocal parabolic equations of the form
$$
(u^q)_t+\mathcal Lu=0,
$$
where $\mathcal L$ is a nonlocal $p$-Laplace-type operator with a symmetric measurable kernel. We consider the fast-diffusion range $1<p<q+1$ and establish Harnack-type estimates in the upper fast-diffusion regime $\widetilde p_c<p<q+1$, where $\widetilde p_c$ is a critical exponent determined by the interplay between the temporal nonlinearity and the nonlocal diffusion.  In particular, the estimates apply to the nonlocal parabolic $p$-Laplace  equation and  the nonlocal porous medium equation.



\end{abstract}

\maketitle

	\begin{center}
		\begin{minipage}{9cm}
			\small
			\tableofcontents
		\end{minipage}
	\end{center}

\addtocontents{toc}{\protect\setcounter{tocdepth}{1}}

\section{Introduction and Main Results}\label{Introduction}

In this paper, we consider a doubly nonlinear nonlocal parabolic equation of the form
\begin{equation}\label{eq1.1}
(|u|^{q-1} u)_t+\mathcal{L}u= 0,\quad (x, t)\in \Omega_T,\quad 1< p < q+1,
\end{equation}
where $\Omega_T:=\Omega \times(0, T)$, $\Omega$ is a bounded open domain in $\mathbb{R}^N$, $u\geqslant 0$ in $\Omega_T$, and the nonlocal operator $\mathcal{L}$ is defined by
\begin{equation}\label{eq1.2}
\mathcal{L} u= \operatorname{P.V.} \int\limits_{\mathbb{R}^N} K(x, y, t) |u(x, t)- u(y, t)|^{p-2}(u(x, t)- u(y, t))\,dy,
\end{equation}
for some $p>1$. Here $\operatorname{P.V.}$ denotes the principal value of the integral, whereas the kernel
$K : \mathbb{R}^N \times \mathbb{R}^N \times (0, T) \rightarrow \mathbb{R}_+$ is a measurable function satisfying the following condition
uniformly in $t$:
\begin{equation}\label{eq1.3}
\frac{K_1}{|x-y|^{N+sp}} \leqslant K(x, y, t)=K(y, x, t)\leqslant \frac{K_2}{|x-y|^{N+sp}},\quad \text{a. e.}\quad x, y \in \mathbb{R}^N, 
\end{equation}
with some $0< K_1 < K_2$ and $s\in (0, 1)$, $sp<N$.  The numbers $K_1$, $K_2$, $p$, $q$, $s$, $N$ are further referred to as the
data. We use $\gamma$ to denote a generic positive constant depending only on the data, whose value may change from line to line.

Before stating the main results, we briefly review the history of the problem. The local counterpart of \eqref{eq1.1} is the following doubly nonlinear parabolic equation:
\begin{equation}\label{eqLocal}
(|u|^{q-1}u)_t-\operatorname{div}(|\nabla u|^{p-2}\nabla u)=0,\qquad p>1,\ q>0.
\end{equation}

The regularity theory for \eqref{eqLocal} is now well developed. In the Trudinger case $q=p-1$, the study of Harnack inequalities dates back to \cite{Tru1968} and was further developed in \cite{GiaVes2006,KinKuu2007}. H\"older continuity of solutions to Trudinger's equation was established in \cite{KuuSilUrb2012,KuuLalSilUrb2012}, and these results were subsequently extended to signed solutions in \cite{BogDuzLia2021}.

For a comprehensive overview of the regularity theory for general doubly nonlinear equations with $p>1$ and $q>0$, including Harnack inequalities and H\"older continuity results in the doubly degenerate, doubly singular, and mixed regimes, we refer the reader to \cite{DiB1993, DiBGiaVes2012, Iva1991, PorVes1993, VesVes2022, Hen2023, LiaSch2022, BogHerSchSin2021, BogDuzGiaLiaSch2026, BogDuzLiaSch2023, CiaHenSavSkrYev2026arxiv, Skr2026arxiv}.

As for the nonlocal equation \eqref{eq1.1}-\eqref{eq1.3}, the Trudinger case $q=p-1$ appears to be the most extensively studied.
These results include local boundedness of solutions, H\"older continuity, Harnack inequalities, and weak Harnack inequalities. For a detailed account of these results, see \cite{BanGarKin2022, BanGarKin2023, CiaNak2026arxiv, Pra2024, ShaZha2026}.   Local boundedness and H\"older continuity of weak solutions have been investigated for both nonlocal parabolic $p$-Laplace equations with $p\neq 2$ and nonlocal porous medium equations; see, for instance, \cite{AdiPraTew2025, BraLinStr2021, ByuKim2024, DeF2026arxiv, DinZhaZho2021, KimLeePra2026, Lia2024a}.   In the singular range $1<p<2$ with $q=1$, integral Harnack-type estimates were obtained in \cite{CasCiaIan2026}.
In the fast-diffusion setting $p=2$ and $q>1$, an integral Harnack inequality was proved for $1<q<\frac{N}{N-2s}$, $N\geq2$, together with an expansion of positivity for $q>1$ under suitable tail assumptions; see \cite{CasDeFSta2026arxiv}. The corresponding integral estimate can be viewed as a step towards a pointwise Harnack theory.

For doubly nonlinear nonlocal parabolic equations with $p>1$ and $q>0$, the local regularity theory is still far from complete.  H\"older continuity of locally bounded weak solutions has been established in the doubly degenerate range $p>2$ and $0<q<p-1$ \cite{Li2026}, as well as in certain mixed regimes \cite{AdiMod2025arxiv}.
An expansion of positivity result under suitable tail conditions was established in \cite{MisYam2025}, following the approach of DiBenedetto, Gianazza and Vespri \cite{DiBGiaVes2012}.  

The literature on Harnack inequalities for these equations appears to be rather limited. A Harnack inequality was established for $q=1$ and $p=2$ in \cite{KasWei2024, LiaWei2025, Str2019} (see also the references therein for an overview of related results).  When $q=1$ and $p>2$, Prasad~\cite{Pra2026arxiv} established an elliptic-type Harnack inequality for the fractional parabolic $p$-Laplacian using a comparison principle. In the same parameter ranges, Liao~\cite{Lia2026arxiv} obtained intrinsic Harnack estimates and established an expansion of positivity result. For $0<q<p-1$, a Harnack inequality was established in \cite{SavSkrYev2026arxiv} based on the ideas developed in \cite{Lia2026arxiv}.  Whether a Harnack inequality holds in the fast diffusion regime remains an open problem. The present paper addresses this issue.

Postponing the formal definitions of weak solution, we will proceed to present the main
results. To this end, we define the following number
\begin{equation}\label{eqnumber}
s\,\widetilde p_c (\widetilde p_c-1)+N(\tilde{p}_c-q-1)=0,\quad \widetilde{p}_c\in \big(p_c, q+1\big),\quad p_c:=\frac{(q+1) N}{N+s q}.
\end{equation}

\begin{theorem}\label{th1.1}
Let $u$ be a continuous, local weak solution to \eqref{eq1.1}-\eqref{eq1.3} in $\Omega_T$, $u\geqslant 0$ in $\Omega_T$ and $\widetilde p_c <p <q+1$. Fix $(x_0,  t_0) \in \Omega_T$ such that $u(x_0, t_0) >0$. There exist constants $\delta \in (0, 1)$, $C_1$, $C_2$, $C_3 >1$ depending only on the data, such that if
$$Q_{16 \rho, 4\theta}(x_0, t_0)\subset \Omega_T,\quad \theta:=[u(x_0, t_0)]^{q-p+1}\rho^{sp},$$
and
\begin{equation}\label{eqthTail}
\bigg(\int\limits_{t_0-\theta}^{t_0+\theta}\,\,\,\int\limits_{\mathbb{R}^N\setminus B_\rho(x_0)}\frac{[u_-(x, \tau)]^{p-1}}{|x-x_0|^{N+sp}}dx d\tau\bigg)^{\frac1q}\leqslant \frac{1}{C_1} u(x_0, t_0),
\end{equation}
then 
\begin{equation}\label{eqHarnack1}
\frac{1}{C_2}\sup\limits_{B_\rho(x_0)\times (t_0-\delta \theta, t_0-\frac{1}{2}\delta \theta)}\,u\leqslant u(x_0, t_0)\leqslant C_2\,\inf\limits_{B_\rho(x_0)\times (t_0+\frac{1}{2}\delta \theta, t_0+\delta\theta)}\, u,
\end{equation}
and 
\begin{equation}\label{eqHarnack2}
\begin{aligned}
\frac{1}{C_3}\bigg(\int\limits^{t_0-\frac{1}{2}\delta \theta}_{t_0-\delta \theta}\,\,\int\limits_{\mathbb{R}^N}\frac{[u_+(x, \tau)]^{p-1}}{(\rho+|x-x_0|)^{N+sp}} dxd\tau\bigg)^{\frac1q}&\leqslant
u(x_0, t_0)\\&\leqslant C_3 \bigg(\,\int\limits^{t_0}_{t_0-\delta \theta}\,\,\,\int\limits_{\mathbb{R}^N}\frac{[u_+(x, \tau)]^{p-1}}{(\rho+|x-x_0|)^{N+sp}} dxd\tau\bigg)^{\frac1q}. 
\end{aligned}
\end{equation}
\end{theorem}




Theorem \ref{th1.1} follows from two propositions, the first of which provides an analogue of a weak Harnack inequality.

\begin{proposition}\label{pr1.1}
Let $u$ be a  local weak super-solution of \eqref{eq1.1}--\eqref{eq1.3} in $\Omega_T$, $u\geqslant 0$ in $\Omega_T$, let $\widetilde p_c < p< q+1$, and fix $\mathcal{I}>0$. Construct the cylinder
$$
Q_{\rho, \theta}(x_0, t_0):=B_{\rho}(x_0)\times (t_0- \theta, t_0+\theta)
\subset Q_{16 \rho, 4\theta}(x_0,  t_0) \subset\Omega_T, \qquad \theta:=\mathcal{I}^{q-p+1}\rho^{sp},
$$
There exist constants  $\delta \in (0, 1)$, $\bar{C}_1$, $\bar{C}_2>1$, depending only on the data, such that if
\begin{equation}\label{eqWeakProp1}
\mathcal{I}\leqslant \frac{1}{\delta^{\frac{N+sp}{spq}}}\left( \,\,\int\limits_{t_0-\delta \theta}^{t_0} \,\,\int\limits_{\mathbb{R}^N}\frac{[u_+(x,t)]^{p-1}}
{(\rho+|x- x_0|)^{N+sp}}dx dt\right)^{\frac1q},
\end{equation}
and
\begin{equation}\label{eqWeakProp2}
\left(\,\int\limits_{t_0-\theta}^{t_0+\theta} \,\,\,\int\limits_{\mathbb{R}^N\setminus B_\rho(x_0)}
\frac{[u_-(x,t)]^{p-1}}{|x-x_0|^{N+sp}}dx dt\right)^{\frac1q}\leqslant \frac{\mathcal{I}}{\bar{C}_1},
\end{equation}
then
\begin{equation}\label{eqWeakProp3}
\mathcal{I}\leqslant \bar{C}_2\inf\limits_{B\rho(x_0)} u(\cdot, t),
\end{equation}
for all $t$ in the range
$$t_0+\frac{1}{2}\delta \theta \leqslant t \leqslant t_0+\delta \theta.$$
\end{proposition}

\begin{remark}\label{remWeak}
Note that since 
$$\int\limits^{t_0}_{t_0-\delta\theta}\,\,\,\int\limits_{B_\rho(x_0)}\frac{\big[u(x, t)\big]^{p-1}}{\big(\rho+|x-x_0|\big)^{N+sp}}\,dx dt\asymp \frac{1}{\rho^{N+sp}} \int\limits^{t_0}_{t_0-\delta \theta}\,\,\,\int\limits_{B_\rho(x_0)}\big[u(x, t)\big]^{p-1}\,dx dt,$$
and
$$\int\limits^{\bar{t}}_{t_0-\delta \theta}\,\,\,\int\limits_{\mathbb{R}^N \setminus B_\rho(x_0)}\frac{\big[u_+(x, t)\big]^{p-1}}{\big(\rho+|x-x_0|\big)^{N+sp}}\,dx dt\asymp  \int\limits^{t_0}_{t_0-\delta \theta}\,\,\,\int\limits_{\mathbb{R}^N\setminus B_\rho(x_0)}\frac{\big[u_+(x, t)\big]^{p-1}}{|x-x_0|^{N+sp}}\,dx dt,$$
condition \eqref{eqWeakProp1} is equivalent (up to the constant) to
\begin{equation}\label{eqWeakProp4}
\mathcal{I}\leqslant \Big(\fint\limits^{t_0}_{t_0-\delta \theta}\,\,\fint\limits_{B_\rho(x_0)}\big[u(x, t)\big]^{p-1}\,dx dt\Big)^{\frac{1}{p-1}}+ \operatorname{Tail}(u_+, Q^-_{\rho, \delta \theta}(x_0, t_0)),\quad \theta=\mathcal{I}^{q-p+1}\rho^{sp}.
\end{equation}
\end{remark}

The second ingredient is a local upper estimate for sub-solutions.

\begin{proposition}\label{pr1.2}
Let $u$ be a  local, weak sub-solution to \eqref{eq1.1}-\eqref{eq1.3} in $\Omega_T$, $u\geqslant 0$ in $\Omega_T$ and  $\widetilde{p}_c < p < q+1$. In the case $0< q< 1$ assume additionally that $u$ is locally bounded in $\Omega_T$. There exists number $\bar{C}_3 > 0$ depending only on the data, such that for every cylinder
$Q^-_{\rho, \theta}(x_0, t_0)\subset \Omega_T$ there holds
\begin{equation}\label{eqUp1}
\begin{aligned}
\sup\limits_{Q^-_{\frac{1}{2}\rho, \frac{1}{2}\theta}(x_0, t_0)} u&\leqslant  \bar{C}_3 \bigg[\Big(\frac{\rho^{sp}}{\theta}\Big)^{\frac{N+sp}{\widetilde \varkappa}}\bigg(\int\limits^{t_0}_{t_0-\theta}\,\,\int\limits_{\mathbb{R}^N} \frac{[u_+(x, \tau)]^{p-1}}{(\rho+|x-x_0|)^{N+sp}}\,dx\,d\tau\bigg)^{\frac{sp}{\widetilde \varkappa}}+
 \Big(\frac{\theta}{\rho^{sp}}\Big)^{\frac{1}{q+1-p}}\\&+ \bigg( \int\limits^{t_0}_{t_0-\theta}\,\,\int\limits_{\mathbb{R}^N}\frac{[u_+(x, \tau)]^{p-1}}{(\rho+|x-x_0|)^{N+sp}}dx d\tau\bigg)^{\frac{1}{q}}\bigg],
 \end{aligned}
\end{equation}
here
$$\widetilde \varkappa:=sp(p-1)+N(p-1-q) > 0.$$
\end{proposition}

The appearance of the exponent $\widetilde p_c$ is quite natural.
Critical exponents play an important role in the theory of nonlinear diffusion equations, as they  separate qualitatively different regimes in the behavior of solutions. For the fractional $p$-Laplacian evolution equation in the fast diffusion range $1< p< 2$, V\'azquez \cite{Vaz2021} introduced a critical exponent $p_1$, characterized as the unique solution of
$$
sp_1(p_1-1)+N(p_1-2)=0.
$$
This exponent determines the transition between two different regimes of spatial decay of the fundamental solution. In particular, the decay rate changes from $r^{-sp/(2-p)}$ for $p < p_1$ to $r^{-(N+sp)}$ for $p> p_1$. A related role is played by critical exponents in the theory of porous-medium-type equations, where they distinguish different regimes of diffusion and the qualitative behavior of solutions \cite{PabQuiRodVaz2011, Vaz2014}.

Motivated by these results, we introduce the critical exponent $\widetilde p_c$ for the doubly nonlinear nonlocal equation considered in this paper. Its role is closely connected with the local boundedness of non-negative subsolutions.  More precisely, the De Giorgi iteration provides a preliminary estimate, while the subsequent interpolation argument requires the positivity of
$$
\widetilde \varkappa =s p(p-1)+N(p-q-1).
$$
This condition is equivalent to $p>\widetilde p_c$ in the range $p<q+1$. Consequently, our approach yields the local upper estimate for subsolutions \eqref{eqUp1}, and hence the integral Harnack-type estimate \eqref{eqHarnack2}. Thus,  the range $\widetilde p_c<p<q+1$
arises naturally in the proof of our main result.  The critical exponent $\widetilde p_c$ also connects the previously studied cases. When $q=1$, its defining equation reduces to that of $p_1$ introduced by V\'azquez. In the  case $p=2$, the critical relation becomes $q< (N+2s)/N$, corresponding, under the formal transformation $v=u^q$, to the  exponent $m_1=N/(N+2s)$ for the fractional porous medium equation. These connections indicate that $\widetilde p_c$ provides a natural critical threshold for the interaction between the temporal nonlinearity and the nonlocal diffusion.  

The validity of our main result for $p_c < p\leqslant \widetilde p_c$, however, remains an open question. It is also unclear  whether an elliptic Harnack inequality analogous to that established by DiBenedetto, Gianazza, and Vespri \cite{DiBGiaVes2012} in the local case  holds for the nonlocal equation considered here.

The paper is organized as follows. Section~\ref{sec:aux} introduces the notation and the definition of a weak solution to \eqref{eq1.1} and presents preliminary results, including local energy estimates and De Giorgi-type lemmas. Section~\ref{sec:Exp} establishes the expansion of positivity, a key ingredient in the proof of the main result. Sections~\ref{sec:Weak} and~\ref{sec:Upper} are devoted to the proofs of Propositions~\ref{pr1.1} and~\ref{pr1.2}, respectively. Finally, Section~\ref{sec:Harnack} combines these estimates to prove Theorem~\ref{th1.1}.

\section{Auxiliary Propositions}\label{sec:aux}

\subsection{Definitions and Notations}

We first introduce some notation that will be used throughout the paper. For any $a$, $b > 0$ we  write $a\asymp b$ if there exists $\gamma >0$ such that
$$\gamma^{-1}\, b \leqslant a \leqslant \gamma\,b.$$
For  
$\rho$,  $\theta >0$ define the cylinders
$$Q^+_{\rho, \theta}(x_0, t_0):=B_\rho(x_0)\times (t_0, t_0+\theta ),\quad Q^-_{\rho, \theta}(x_0, t_0):=B_\rho(x_0)\times (t_0-\theta, t_0),$$
and
$$Q_{\rho, \theta}(x_0, t_0):=Q^-_{\rho, \theta}(x_0, t_0)\cup Q^+_{\rho, \theta}(x_0, t_0).$$

\subsubsection{\textbf{Function Spaces}}

Let $p>1$ and $s\in(0,1)$. We define the fractional Sobolev space $W^{s,p}(\mathbb{R}^N)$ by
$$W^{s, p}(\mathbb{R}^N):=\big\{ u\in L^p(\mathbb{R}^N) : \int\limits_{\mathbb{R}^N}\int\limits_{\mathbb{R}^N}\frac{|u(x)-u(y)|^p}{|x-y|^{N+sp}}\,dx dy < \infty\big\},$$
equipped with the norm 
$$||u||_{W^{s, p}(\mathbb{R}^N)}:= \Big(\int\limits_{\mathbb{R}^N}|u|^p dx \Big)^{\frac{1}{p}}+
\bigg(\int\limits_{\mathbb{R}^N}\int\limits_{\mathbb{R}^N}\frac{|u(x)-u(y)|^p}{|x-y|^{N+sp}}\,dx dy\bigg)^{\frac{1}{p}}.$$
For a domain $\Omega\subset\mathbb{R}^N$, the corresponding fractional
Sobolev space $W^{s, p}(\Omega)$ is defined analogously. We also denote
$$W^{s, p}_0(\Omega):=\big\{ u\in W^{s, p} (\mathbb{R}^N): u=0 \quad \text{a.e. in}\quad\mathbb{R}^N\setminus\Omega\big\}.$$
In what follows, we need the following parabolic embedding, which is a standard consequence of the elliptic Sobolev embedding theorem (see e.g. \cite[Proposition A.3]{Lia2024}).

\begin{lemma}\label{lemEmbedding}
Let $p>1$ and $s\in(0,1)$, with $sp<N$. For any function $u\in L^p(t_1, t_2; W^{s, p}(B_r(x_0)))\cap L^\infty (t_1, t_2; L^{2}(B_r(x_0)))$, compactly supported 
in $B_{r(1-\sigma)}(x_0)$ with some $\sigma \in (0, 1)$ and for all $t\in (t_1, t_2)$, there holds
\begin{equation*}
\begin{aligned}
&\int\limits^{t_2}_{t_1}\int\limits_{B_r(x_0)}|u|^{p\left(1+\frac{2 s}{N}\right)}dx dt
\leqslant 
\gamma \biggr(\sup\limits_{t\in (t_1, t_2)}\int\limits_{B_r(x_0)}|u|^{2}\,dx\biggr)^{\frac{sp}{N}}
\\&\times
\biggr(\int\limits^{t_2}_{t_1}\int\limits_{B_r(x_0)}\int\limits_{B_r(x_0)}
\frac{|u(x, t)-u(y, t)|^p}{|x-y|^{N+sp}}dx dy dt+\frac{1}{\sigma^{N+sp} r^{sp}}\int\limits^{t_2}_{t_1}\int\limits_{B_r(x_0)}|u|^p dx dt\biggr),
\end{aligned}
\end{equation*}
where $\gamma>0$ depends only on $N$, $s$ and  $p$.
\end{lemma}

\subsubsection{\textbf{Notion of Weak Solution}}

\begin{definition}
A measurable function
$u:\mathbb{R}^N\times(0,T]\rightarrow\mathbb{R}$ satisfying
\begin{equation}\label{eqDefinition1}
u\in C_{loc}(0, T; L^{q+1}_{loc}(\Omega))\cap L^p_{loc}(0, T; W^{s, p}_{loc}(\Omega))
\end{equation}
is called a local weak sub(super)-solution to \eqref{eq1.1}, with
the kernel $K$ satisfying \eqref{eq1.3}, if the following conditions
hold. For every compact subset $E\subset\Omega$ and every
subinterval $[t_1,t_2]\subset(0,T]$, we have
\begin{equation}\label{eqDefinition2}
\int\limits_{t_1}^{t_2}\int\limits_{\mathbb{R}^N}\frac{|u(x, t)|^{p-1}}{1+|x|^{N+sp}}\,dx dt <\infty,
\end{equation}
and
\begin{equation}\label{eqDefinition3}
\begin{aligned}
&\int\limits_{E}|u|^{q-1} u\varphi dx\Big|^{t_2}_{t_1}-\int\limits^{t_2}_{t_1}\int\limits_{E}|u|^{q-1}u \varphi_t\,dx dt
\\&+\int\limits^{t_2}_{t_1}\int\limits_{\mathbb{R}^N}\int\limits_{\mathbb{R}^N}K(x, y, t) |u(x, t)-u(y, t)|^{p-2}(u(x, t)-u(y, t))(\varphi(x, t)-\varphi(y, t))\,dx dy dt\leqslant ( \geqslant) 0,
\end{aligned}
\end{equation}
for every nonnegative test function
$$\varphi \in W^{1,q+1}_{loc}(0, T; L^{q+1}(E))\cap L^p_{loc}(0, T; W^{s, p}_0(E)).$$

A function that is both a local weak sub-solution and a local weak
super-solution is called a local weak solution.
\end{definition}

We also introduce the following tail term, which naturally arises in the nonlocal setting:
\begin{equation}\label{eqTail}
\operatorname{Tail}\big(u, Q_{\rho, \theta}(x_0, t_0)\big):= \biggr(\int\limits^{t_0+\theta}_{t_0-\theta}\,\,\int\limits_{\mathbb{R}^N\setminus B_\rho(x_0)}\frac{|u(x, t)|^{p-1}}{|x-x_0|^{N+sp}}\,dx\,dt\biggr)^{\frac{1}{q}},
\end{equation}

\subsection{Algebraic Lemma}

In what follows, we will need the following inequality
\begin{lemma}\label{lemAlgebraic}
There hold
\begin{equation}\label{eqalg1}
 ||a|^{p-1}a -|b|^{p-1} b|\asymp \big(|a|+|b|\big)^{p-1}|a-b|,\quad p>0, \quad a, b\in \mathbb{R},
 \end{equation}

\end{lemma}

\subsection{Energy Estimates}

We refer the reader to \cite[Lemma 2.4]{SavSkrYev2026arxiv} for the proof of the following lemma.

\begin{lemma}\label{lemEnergy}
Let $u$ be a local weak sub(super)-solution to \eqref{eq1.1}-\eqref{eq1.3} in $\Omega_T$ and let $k(t)$ be absolutely continuous in $(0, T)$.
Then for every cylinder $Q^-_{8r, 8\eta}:=Q^-_{8r, 8\eta}(x_0, t_0)\subset \Omega_T$ and every piecewise smooth cutoff function $\zeta(x, t)$ compactly supported in $B_r:=B_r(x_0)$ for all
$t\in (t_0-\eta, t_0)$ there holds
\begin{equation}\label{eqenergy}
\begin{aligned}
\int\limits_{B_r\times\{\tau\}}\!\!\!g(u^{\pm}_{k(\tau)})\zeta^p\,dx\Big|^{t_0}_{t_0-\eta}&+\int\limits_{t_0-\eta}^{t_0}\,\int\limits_{B_r}\int\limits_{B_r}
\frac{|u^{\pm}_{k(\tau)}(x, \tau)-u^{\pm}_{k(\tau)}(y, \tau)|^p}{|x-y|^{N+sp}}\min[\zeta(x, \tau), \zeta(y, \tau)]^p\,dx dy d\tau\\
&+\int\limits^{t_0}_{t_0-\eta}\,\int\limits_{B_r}u^{\pm}_{k(\tau)}(x, \tau)\zeta^p(x, \tau)\,\bigg(\int\limits_{B_r}\frac{\big[u^{\mp}_{k(\tau)}(y, \tau)\big]^{p-1}}{|x-y|^{N+sp}}dy\bigg)\,dx\,d\tau\\
&\leqslant \mp\iint\limits_{Q_{r, \eta}}k'(\tau)h(u^{\pm}_{k(\tau)})\zeta^p(x, \tau) dxd\tau+\gamma \iint\limits_{Q_{r, \eta}}g(u^{\pm}_{k(\tau)})\zeta^{p-1}(x, \tau)|\zeta_\tau| dxd\tau\\
&+\gamma\int\limits^{t_0}_{t_0-\eta}\,\int\limits_{B_r}\int\limits_{B_r}\max[u^{\pm}_{k(\tau)}(x, \tau), u^{\pm}_{k(\tau)}(y, \tau)]^p\frac{|\zeta(x, \tau)-\zeta(y, \tau)|^p}{|x-y|^{N+sp}}\,dxdyd\tau\\
&+\int\limits^{t_0}_{t_0-\eta}\,\int\limits_{B_r}u^{\pm}_{k(\tau)}(x, \tau)\zeta^p(x, \tau)\bigg(\int\limits_{\mathbb{R}^N\setminus B_r}\frac{\big[u^{\pm}_{k(\tau)}(y, \tau)\big]^{p-1}}{|x-y|^{N+sp}} dy\bigg) dx d\tau.
\end{aligned}
\end{equation}
Here
$$u^{\pm}_{k(\tau)}:=(u-k(\tau))_{\pm},\quad g(u^{\pm}_{k(\tau)}):=\int\limits^{u^{\pm}_{k(\tau)}}_0|z\pm k(\tau)|^{q-1}\,z dz,\quad h(u^{\pm}_{k(\tau)}):=\int\limits^{u^{\pm}_{k(\tau)}}_0|z\pm k(\tau)|^{q-1} dz.$$

\end{lemma}

\subsection{De Giorgi Type Lemma}
Construct the cylinder 
$$\mathcal{Q}:=B_R(x_0)\times (T_1, T_2).$$
The following lemma can be found, for example, in \cite[Lemma 2.5]{SavSkrYev2026arxiv}. For the case $q=1$ we refer the reader to \cite[Lemma 3.1]{Lia2024a}.

\begin{lemma}\label{lemDeGiorgi}
Let $u$ be a  local weak super-solution to \eqref{eq1.1}-\eqref{eq1.3} with $q> 0$, $p> 1$, $u\geqslant 0$ in $\Omega_T$. Fix $k>0$ and let $\rho\in (0, \frac{1}{2}R)$, $\theta >0$ be the numbers such that $Q_{\rho, \theta}:=Q^-_{\rho, \theta}(x_0, t_0) \subset \mathcal{Q}$. There exist numbers $C>1$, $\nu \in (0, 1)$ depending only on the data and $r$, $\theta$, $k$ such that if
\begin{equation}\label{eqDeGiorgi2}
\big|Q_{\rho, \theta}\cap \big\{u \leqslant k\big\}\big|\leqslant \nu |Q_{\rho, \theta}|,
\end{equation}
then either
\begin{equation}\label{eqDeGiorgi3}
C\,\operatorname{Tail}(u_{-}, \mathcal{Q})\geqslant k,
\end{equation}
or
\begin{equation}\label{eqDeGiorgi4}
u \geqslant \frac{1}{4}\,k,\quad \text{in}\quad Q_{\frac{1}{2}\rho, \frac{1}{2}\theta}.
\end{equation}
\end{lemma}
The number $\nu \in (0, 1)$ can be  chosen to satisfy
$$\nu:=\frac{1}{\gamma}k^{q+1-p}\frac{\rho^{sp}}{\theta}\bigg[1+k^{q+1-p}\frac{\rho^{sp}}{\theta}\bigg]^{-1-\frac{N}{sp}}.$$

\section{Expansion of Positivity}\label{sec:Exp}

\subsection{Propagation of Positivity in Measure}

The following lemma propagates measure theoretical information forward in time, for the proof we refer the reader to \cite[Lemma 3.1]{SavSkrYev2026arxiv}.

\begin{lemma}\label{lemPropagation}
Let $u$ be a  local weak super-solution to \eqref{eq1.1}- \eqref{eq1.3}, $u \geqslant 0$ in $\Omega_T$, $q>0$, $p>1$. 
Suppose also  that
\begin{equation}\label{eqProp2}
\big|B_r(x_0)\cap\big\{u(\cdot, t_0)\geqslant k \big\}\big|\geqslant \alpha |B_r(x_0)|,
\end{equation}
with some $k > 0$ and some $\alpha \in (0, 1)$. Then there exist constants $\tilde{\epsilon}$, $\tilde{\delta} \in (0, 1)$ depending
only on the data and $\alpha$ such that
either
\begin{equation}\label{eqProp3}
\,\operatorname{Tail}(u_{-}, \mathcal{Q}) > \tilde{\delta}^{\frac{1}{q}}\,k,
\end{equation}
or
\begin{equation}\label{eqProp4}
\big|B_r(x_0)\cap\big\{u(\cdot, t))\geqslant \tilde{\epsilon}\,k\big\}\big|\geqslant \frac{1}{4}\alpha |B_r(x_0)|,
\quad\text{for all}\quad  t\in(t_0, t_0+ \tilde{\delta}\,(\tilde \epsilon k)^{q-p+1}\,r^{sp}), 
\end{equation}
provided that  $B_{4r}(x_0)\times (t_0, t_0+ (\tilde \epsilon k)^{q-p+1}\,r^{sp})\subset \mathcal{Q}$.

\end{lemma}

\subsection{Two Versions of the Measure Shrinking Lemma}

The following measure shrinking lemmas can  be readily derived.  The argument runs  like in \cite[Lemma 3.4]{Lia2024a} and \cite[Lemma 3.6, Lemma 3.7]{Lia2026arxiv}.

\begin{lemma}\label{lemShrinking}
Let $u$ be a  local weak super-solution to \eqref{eq1.1}-\eqref{eq1.3}, $u \geqslant 0$ in $\Omega_T$, $q>0$, $p>1$, and let
 also  with some  $\delta \in \big(0, 1\big)$, $\epsilon \in (0, 1/2)$ and $k >0$ 
there holds
\begin{equation}\label{eqShr2}
\big|B_{\frac{3}{4} r}(x_0)\cap\big\{u(\cdot, t))\geqslant k\big\}\big|\geqslant \alpha |B_{\frac{3}{4} r}(x_0)|,
\end{equation}
for all $t\in (t_0-\delta (\epsilon k)^{q-p+1} r^{sp}, t_0)$. Then, either
\begin{equation}\label{eqShr3}
\operatorname{Tail}(u_{-}, \mathcal{Q}) > \epsilon\,k,
\end{equation}
or 
\begin{equation}\label{eqShr4}
\big|Q^-_{\frac{3}{4} r, \delta \eta}(x_0, t_0)\cap\big\{u \leqslant \frac{1}{2}\epsilon k\big\}\big|\leqslant \frac{\gamma\,\epsilon^{p-1}}{\delta\,\alpha}|Q^-_{\frac{3}{4} r, \delta \eta}(x_0, t_0)|,\quad \eta=(\epsilon k)^{q-p+1}r^{sp},
\end{equation}
provided that $Q^-_{4 r, \eta}(x_0, t_0)\subset \mathcal{Q}$.
\end{lemma}
\begin{proof}
Use  energy estimate \eqref{eqenergy} for the function $u^-_{k_\epsilon}$ with $k_{\epsilon}=\epsilon k$. The cutoff function $\zeta(x)$ is such that $\zeta(x)\in C^1_0(B_{\frac{7}{8}r}(x_0))$, $\zeta(x)=1$ in $B_{\frac{3}{4}r}(x_0)$, $0\leqslant \zeta(x)\leqslant 1$ and $|\nabla \zeta(x)|\leqslant \frac{8}{r}$. We obtain 
\begin{equation}\label{eqS1}
\begin{aligned}
\int\limits^{t_0}_{t_0-\delta \eta}\,\,\int\limits_{B_{\frac{3}{4} r}(x_0)}&u^-_{k_\epsilon}(x, \tau)\,\bigg(\int\limits_{B_{\frac{3}{4}r}(x_0)}\frac{\big[u^+_{k_\epsilon}(y, \tau)\big]^{p-1}}{|x-y|^{N+sp}}dy\bigg)\,dx\,d\tau\\&\leqslant \gamma \int\limits_{B_{r}(x_0)\times\{t_0-\delta \eta\}}\,\,\int\limits_u^{k_\epsilon}\,z^{q-1}(k_\epsilon-z) dz  dx\\&+
\gamma\int\limits^{t_0}_{t_0-\delta \eta}\,\,\int\limits_{B_r(x_0)}\int\limits_{B_r(x_0)}\max[u^-_{k_\epsilon}(x, \tau), u^-_{k_\epsilon}(y, \tau)]^p\frac{|\zeta(x)-\zeta(y)|^p}{|x-y|^{N+sp}}\,dxdyd\tau\\&+
\int\limits^{t_0}_{t_0-\delta \eta}\,\,\int\limits_{B_{\frac{7}{8}r}(x_0)}u^-_{k_\epsilon}(x, \tau)\bigg(\sup\limits_{x\in B_{\frac{7}{8} r}(x_0)}\int\limits_{\mathbb{R}^N\setminus B_r(x_0)}\frac{\big[ u^-_{k_\epsilon}(y, \tau)\big]^{p-1}}{|x-y|^{N+sp}} dy\bigg) dx d\tau.
\end{aligned}
\end{equation}
We estimate the integrals on the right-hand side of \eqref{eqS1} as follows
$$
\int\limits_{B_{r}(x_0)\times\{t_0-\delta \eta\}}\,\,\int\limits_u^{k_\epsilon}\,z^{q-1}(k_\epsilon-z) dz  dx\leqslant \gamma (\epsilon k)^{q+1}|B_r(x_0)|=\frac{\gamma (\epsilon k)^p}{\delta r^{sp}} |Q^-_{r, \delta \eta}(x_0, t_0)|.
$$
Using the Lipschitz continuity of the  cutoff function,
$$|\zeta(x)-\zeta(y)|\leqslant \gamma \frac{|x-y|}{r},$$
we obtain
\begin{equation*}
\int\limits^{t_0}_{t_0-\delta \eta}\,\,\int\limits_{B_r(x_0)}\int\limits_{B_r(x_0)}\max[u^-_{k_\epsilon}(x, \tau), u^-_{k_\epsilon}(y, \tau)]^p\frac{|\zeta(x)-\zeta(y)|^p}{|x-y|^{N+sp}}\,dxdyd\tau\leqslant \frac{\gamma (\epsilon k)^p}{r^{sp}} |Q^-_{r, \delta \eta}(x_0, t_0)|.
\end{equation*}
Assuming that \eqref{eqShr3} is violated, we estimate the last term on the right-hand side of \eqref{eqS1} as follows
\begin{equation*}
    \begin{aligned}
\int\limits^{t_0}_{t_0-\delta \eta}\,\,\int\limits_{B_{\frac{7}{8}r}(x_0)}&u^-_{k_\epsilon}(x, \tau)\bigg(\sup\limits_{x\in B_{\frac{7}{8} r}(x_0)}\,\,\int\limits_{\mathbb{R}^N\setminus B_r(x_0)}\frac{\big[ u^-_{k_\epsilon}(y, \tau)\big]^{p-1}}{|x-y|^{N+sp}} dy\bigg) dx d\tau\\
&\leqslant \int\limits^{t_0}_{t_0-\delta \eta}\,\,\int\limits_{B_{\frac{7}{8}r}(x_0)}u^-_{k_\epsilon}(x, \tau)\bigg(\sup\limits_{x\in B_{\frac{7}{8} r}(x_0)}\,\,\int\limits_{B_R(x_0)\setminus B_r(x_0)}\frac{\big[ u^-_{k_\epsilon}(y, \tau)\big]^{p-1}}{|x-y|^{N+sp}} dy\bigg) dx d\tau\\
&+\int\limits^{t_0}_{t_0-\delta \eta}\,\,\int\limits_{B_{\frac{7}{8}r}(x_0)}u^-_{k_\epsilon}(x, \tau)\bigg(\sup\limits_{x\in B_{\frac{7}{8} r}(x_0)}\,\,\int\limits_{\mathbb{R}^N\setminus B_R(x_0)}\frac{\big[ u^-_{k_\epsilon}(y, \tau)\big]^{p-1}}{|x-y|^{N+sp}} dy\bigg) dx d\tau\\
&\leqslant \frac{\gamma(\epsilon k)^p}{r^{sp}}|Q^-_{r, \delta \eta}(x_0, t_0)|+\gamma \epsilon k\big[\operatorname{Tail}(u_-, \mathcal{Q})\big]^q|B_r(x_0)|\leqslant \frac{\gamma (\epsilon k)^p}{\delta r^{sp}}|Q^-_{r, \delta \eta}(x_0, t_0)|.
\end{aligned}
\end{equation*}
Thus, the last three inequalities and \eqref{eqS1} yield
$$
\int\limits^{t_0}_{t_0-\delta \eta}\int\limits_{B_{\frac{3}{4} r}(x_0)}u^-_{k_\epsilon}(x, \tau)\,\bigg(\int\limits_{B_{\frac{3}{4}r}(x_0)}\frac{\big[u^+_{k_\epsilon}(y, \tau)\big]^{p-1}}{|x-y|^{N+sp}}dy\bigg)\,dx\,d\tau\leqslant 
\frac{\gamma (\epsilon k)^p}{\delta r^{sp}}|Q^-_{r, \delta \eta}(x_0, t_0)|,
$$
provided that \eqref{eqShr3} is violated.
The left-hand side of this inequality we estimate  using \eqref{eqShr2}
\begin{equation*}
\alpha\,\epsilon k^p\big|Q^-_{\frac{3}{4} r, \delta \eta}(x_0, t_0)\cap\big\{u\leqslant \frac{1}{2} \epsilon\,k\big\}\big|\leqslant \frac{\gamma}{\delta}(\epsilon k)^p |Q^-_{\frac{3}{4} r, \delta\eta}(x_0, t_0)|,
\end{equation*}
which yields
$$
\big|Q^-_{\frac{3}{4} r, \delta \eta}(x_0, t_0)\cap\big\{u\leqslant \frac{1}{2} \epsilon\,k\big\}\big|\leqslant \frac{\gamma \epsilon^{p-1}}{\delta \alpha} |Q^-_{\frac{3}{4} r, \delta\eta}(x_0, t_0)|.
$$
This completes the proof of the lemma.
\end{proof}

\begin{lemma}\label{lemShrinking1}
Let $u$ be a  local weak super-solution to \eqref{eq1.1}- \eqref{eq1.3}, $u \geqslant 0$ in $\Omega_T$, $q>p-1$, $p>1$, and fix  $k > 0$ and $\delta \in \big(0, 1\big)$. 
 For every $\epsilon \in (0, \frac{1}{2})$ there exists $\gamma >0$ depending only on the data, such that either
\begin{equation}\label{eqShr6}
\operatorname{Tail}(u_{-}, \mathcal{Q}) > \epsilon\,k,
\end{equation}
or there exists a time level $\tilde{t}\in (t_0-\delta k^{q-p+1} r^{sp}, t_0)$ such that
\begin{equation}\label{eqShr7}
\big|B_{\frac{3}{4} r}(x_0)\cap\big\{u(\cdot, \tilde{t})\leqslant \frac{1}{2} \epsilon k\big\}\big|\leqslant \frac{\gamma \big[\epsilon^q+\epsilon^{p-1}\delta\big] k^q}{\int\limits^{t_0}_{t_0-\delta \eta}\,\,\int\limits_{\mathbb{R}^N}\frac{[u_+(y, \tau)]^{p-1}}{(r+|x_0-y|)^{N+sp}}dy d\tau}|B_{\frac{3}{4} r}(x_0)|,
\end{equation}
provided that $Q^-_{4 r, \eta}(x_0, t_0)\subset \mathcal{Q}$, $\eta:=k^{q-p+1} r^{sp}$.
\end{lemma}

\begin{proof}
We use \eqref{eqenergy} with $k_\epsilon =\epsilon k$ and the cutoff function is chosen as in the previous lemma. Estimating the integrals on the right-hand side as in Lemma \ref{lemShrinking}, we obtain
\begin{equation*}
\begin{aligned}
\int\limits^{t_0}_{t_0-\delta \eta}\,\,&\int\limits_{B_{\frac{3}{4} r}(x_0)}u^-_{k_\epsilon}(x, \tau)\,\bigg(\int\limits_{\mathbb{R}^N}\frac{\big[u^+_{k_\epsilon}(y, \tau)\big]^{p-1}}{|x-y|^{N+sp}}dy\bigg)\,dx\,d\tau \\&\leqslant \gamma \int\limits_{B_{r}(x_0)\times\{t_0-\delta \eta\}}\,\,\int\limits_u^{k_\epsilon}\,z^{q-1}(k_\epsilon-z) dz  dx  \\&+
\gamma\int\limits^{t_0}_{t_0-\delta \eta}\,\,\int\limits_{B_r(x_0)}\int\limits_{B_r(x_0)}\max[u^-_{k_\epsilon}(x, \tau), u^-_{k_\epsilon}(y, \tau)]^p\frac{|\zeta(x, \tau)-\zeta(y, \tau)|^p}{|x-y|^{N+sp}}\,dxdyd\tau \\&+
\int\limits^{t_0}_{t_0-\delta \eta}\,\,\int\limits_{B_{\frac{7}{8}r}(x_0)}u^-_{k_\epsilon}(x, \tau)\bigg(\sup\limits_{x\in B_{\frac{7}{8} r}(x_0)}\int\limits_{\mathbb{R}^N\setminus B_r(x_0)}\frac{\big[ u^-_{k_\epsilon}(y, \tau)\big]^{p-1}}{|x-y|^{N+sp}} dy\bigg) dx d\tau\leqslant
 \gamma (\epsilon k)^{q+1}|B_r(x_0)|\\&+ \frac{\gamma (\epsilon k)^p}{r^{sp}} |Q^-_{r, \delta \eta}(x_0, t_0)|\leqslant \gamma \Big[\epsilon^{q+1}+\epsilon^p\,\delta\Big] k^{q+1}|B_r(x_0)|,
\end{aligned}
\end{equation*}
provided that \eqref{eqShr6} is violated.  Next, we have
\begin{multline*}
\int\limits^{t_0}_{t_0-\delta \eta}\,\,\int\limits_{B_{\frac{3}{4} r}(x_0)}u^-_{k_\epsilon}(x, \tau)\,\bigg(\int\limits_{\mathbb{R}^N}\frac{\big[u^+_{k_\epsilon}(y, \tau)\big]^{p-1}}{|x-y|^{N+sp}}dy\bigg)\,dx\,d\tau \\\geqslant
\gamma \epsilon k \,\inf\limits_{t\in (t_0-\delta \eta, t_0)} \big|B_{\frac{3}{4} r}(x_0)\cap\big\{u(\cdot, t) \leqslant \frac{1}{2}\epsilon k\big\}\big|\int\limits^{t_0}_{t_0-\delta \eta}\,\,\int\limits_{\mathbb{R}^N}\frac{\big[u^+_{k_\epsilon}(y, \tau)\big]^{p-1}}{\big(r+|x_0-y|\big)^{N+sp}}dy d\tau.
\end{multline*}
By direct calculations
\begin{equation*}
\int\limits^{t_0}_{t_0-\delta \eta}\,\,\int\limits_{\mathbb{R}^N}\frac{\big[u^+_{k_\epsilon}(y, \tau)\big]^{p-1}}{\big(r+|x_0-y|\big)^{N+sp}}dy d\tau\geqslant  \gamma \int\limits^{t_0}_{t_0-\delta \eta}\,\,\int\limits_{\mathbb{R}^N}\frac{\big[u_+(y, \tau)\big]^{p-1}}{\big(r+|x_0-y|\big)^{N+sp}}dy d\tau -\gamma (\epsilon k)^q.
\end{equation*}
Combining the last three inequalities, we arrive at
\begin{equation}\label{eqShr8}
\inf\limits_{t\in (t_0-\delta \eta, t_0)} \big|B_{\frac{3}{4} r}(x_0)\cap\big\{u(\cdot, t)\leqslant\frac{1}{2}\epsilon k\big\}\big| \leqslant\frac{\gamma \big[\epsilon^q +\epsilon^{p-1}\delta\big] k^q}{\int\limits^{t_0}_{t_0-\delta \eta}\,\,\int\limits_{\mathbb{R}^N}\frac{[u_+(y, \tau)]^{p-1}}{(r+|x_0-y|)^{N+sp}}dy d\tau}|B_r(x_0)|,
\end{equation}
which completes the proof of the lemma.
\end{proof}

\subsection{Expansion of Positivity}

The following proposition will be used in the sequel. 

\begin{proposition}\label{pr.Expansion}
Let $u$ be a  local weak super-solution to \eqref{eq1.1}--\eqref{eq1.3} in $\Omega_T$, $u\geqslant 0$ in $\Omega_T$, and let $q> p-1$. Fix $k >0$, construct the cylinder $Q^-_{4r, k^{q-p+1}r^{sp}}(x_0, t_0)\subset \mathcal{Q}$ and assume that
\begin{equation}\label{eqE1}
|B_{\frac{3}{4} r}(x_0)\cap\big\{u(\cdot, t)\geqslant \xi\,k\big\}|\geqslant \alpha|B_{\frac{3}{4} r}(x_0)|,
\end{equation}
for all $t\in (t_0-\delta\,k^{q-p+1} r^{sp}, t_0)$, with some  $\delta$,
$\xi$, $\alpha \in (0, 1)$. Then there exist numbers $\epsilon_*$, $\xi_*\in (0, 1)$, depending only on the data,  $\alpha$ and $\delta$, such that either
\begin{equation}\label{eqE2}
\operatorname{Tail}(u_{-}, \mathcal{Q}) > \frac{1}{4}\,\epsilon_*\,k,
\end{equation}
or 
\begin{equation}\label{eqE3}
 u(x, t)\geqslant \frac{1}{4}\xi\,\xi_*\,k,\quad  (x, t)\in B_{\frac{r}{2}}(x_0)\times (t_0-\frac{1}{4}\,\delta\,k^{q-p+1} r^{sp}, t_0).
\end{equation}
\end{proposition}

\begin{proof}
Inside the cylinder $Q^-_{r,\delta k^{q-p+1}r^{sp}}(x_0,t_0)$, we consider a smaller cylinder $Q^-_{r,\eta}(x_0,\bar t)$, where $\eta=\delta(\xi \xi_* k)^{q-p+1}r^{sp}$,   $\xi_*\in\ (0,\frac14)$ is to be chosen, and  $\bar t\in (
t_0-\delta\big(1-(\xi \xi_*)^{q-p+1}\big)k^{q-p+1}r^{sp}, t_0)$
is arbitrary.  We apply Lemma \ref{lemShrinking} in each of the cylinders $Q^-_{r, \eta}(x_0, \bar{t})$ and obtain
$$\big|Q^-_{\frac{3}{4} r, \delta \eta}(x_0, \bar{t})\cap\big\{u \leqslant \frac{1}{2}\xi \xi_* k\big\}\big|\leqslant \frac{\gamma\,(\xi \xi_*)^{p-1}}{\delta\,\alpha}|Q^-_{\frac{3}{4} r, \delta \eta}(x_0, \bar{t})|.$$
Using De Giorgi type Lemma \ref{lemDeGiorgi} and choosing $\xi_*$ such that $\frac{\gamma (\xi \xi_*)^{p-1}}{\delta\,\alpha}=\nu$, we conclude that
$$u\geqslant \frac{1}{4}\xi\,\xi_*\,k, \quad \text{in}\quad B_{\frac{1}{2}r}(x_0)\times (\bar{t}-\frac{1}{2}\eta, \bar{t}).$$
Since $\bar{t}$ is arbitrary in the interval $(t_0-\delta k^{q-p+1} r^{sp}+ \delta(\xi \xi_*k)^{q-p+1} r^{sp}, t_0)$, this completes the proof of the proposition.

 \end{proof}

\section{Proof of Proposition \ref{pr1.1}}\label{sec:Weak}

We recall the definition of $\mathcal{I}$
\begin{equation}\label{eqPR1}
\mathcal{I}\leqslant \frac{1}{\delta^{\frac{N+sp}{spq}}}\bigg(\int\limits^{t_0}_{t_0-\delta\theta}\,\,\int\limits_{\mathbb{R}^N}\frac{\big[u_+(x, t)\big]^{p-1}}{\big(\rho+|x-x_0|\big)^{N+sp}}\,dx dt\bigg)^{\frac{1}{q}},\quad \theta=\mathcal{I}^{q-p+1} \rho^{sp}
\end{equation}
with some  $\delta \in (0, 1)$.
To prove Proposition \ref{pr1.1} we use the measure shrinking lemma, Lemma \ref{lemShrinking1} with $k$ replaced by $\mathcal{I}$, which by \eqref{eqPR1} yields, for every $\epsilon \in (0, 1/2)$ that
\begin{equation}\label{eqPR2}
\big|B_{\frac{3}{4} \rho}(x_0)\cap\big\{u(\cdot, \tilde{t})\leqslant \frac{1}{2} \epsilon \mathcal{I}\big\}\big|\leqslant \frac{\gamma\big[\epsilon^q+ \epsilon^{p-1}\delta\big]}{\delta^{\frac{N+sp}{sp}}}|B_{\frac{3}{4} \rho}(x_0)|,
\end{equation}
for some time level $\tilde{t}\in (t_0-\delta \theta, t_0)$, provided that
$$\operatorname{Tail}(u_-, Q_{\rho, \theta}(x_0, t_0))\leqslant \frac{1}{\gamma}\epsilon\,\mathcal{I}.$$
Choose $\epsilon $ sufficiently small so that
$$\epsilon=\min \Big[\Big(\frac{1}{8\gamma} \delta^{\frac{N+sp}{sp}}\Big)^{\frac1q}, \Big(\frac{1}{8\gamma}\delta^{\frac{N}{sp}}\Big)^{\frac{1}{p-1}}\Big]=\Big(\frac{1}{8\gamma} \delta^{\frac{N+sp}{sp}}\Big)^{\frac1q},$$
consequently, \eqref{eqPR2} yields
\begin{equation*}
\big|B_{\frac{3}{4} \rho}(x_0)\cap\big\{u(\cdot, \tilde{t})\leqslant \frac{1}{2} \epsilon \mathcal{I}\big\}\big|\leqslant \frac{1}{4} |B_{\frac{3}{4} \rho}(x_0)|.
\end{equation*}
Lemma \ref{lemPropagation} ensures the existence of $\tilde{\epsilon}$, $\tilde{\delta} \in (0, 1)$ depending only on the data such that
\begin{equation}\label{eqW1}
\big|B_{\frac{3}{4} \rho}(x_0)\cap\big\{u(\cdot, t)\leqslant \tilde{\epsilon}\epsilon \mathcal{I}\big\}\big|\leqslant \frac{1}{2}|B_{\frac{3}{4} \rho}(x_0)|,
\end{equation}
for all time levels $\tilde{t}\leqslant t \leqslant \tilde{t}+\tilde{\delta}(\tilde{\epsilon}\epsilon \mathcal{I})^{q-p+1}\rho^{sp}.$
Choose $\delta$ smaller if necessary, such that
$$(\tilde{t}, \tilde{t}+\tilde{\delta}(\tilde{\epsilon}\epsilon \mathcal{I})^{q-p+1}\rho^{sp})\supset(t_0, t_0+\delta \mathcal{I}^{q-p+1}\rho^{sp}),$$
in view of our choice of $\epsilon$, that is
$$2\delta\leqslant \frac{1}{\gamma}\tilde{\delta}\tilde{\epsilon}^{q-p+1}\,\delta^{\frac{(N+sp)(q-p+1)}{spq}},
\quad \text{i. e.}\quad \delta\leqslant \bigg(\frac{1}{\gamma}\tilde{\delta}\tilde{\epsilon}^{q-p+1}\bigg)^{\frac{spq}{\widetilde  \varkappa}}.$$
This yields \eqref{eqW1} for all $t$ in the range $t_0\leqslant t \leqslant t_0+\delta \theta$. Using Proposition \ref{pr.Expansion}, we obtain 
$$\inf\limits_{B_{\frac{1}{2}\rho}(x_0)} u(\cdot, t)\geqslant \tilde \epsilon \epsilon \xi_* \,\mathcal{I},$$
with some $\xi_*\in (0, 1)$ depending only on the data and for all time levels
$$t \in (t_0+\frac{1}{2}\delta \mathcal{I}^{q-p+1} \rho^{sp}, t_0+\delta \mathcal{I}^{q-p+1} \rho^{sp}).$$
Thus, \eqref{eqWeakProp3} follows, concluding the proof of Proposition \ref{pr1.1}.

\section{Proof of Proposition \ref{pr1.2}}\label{sec:Upper}

In this section we prove the upper bound for  sub-solutions to \eqref{eq1.1}-\eqref{eq1.3}, $u \geqslant 0$ in $\Omega_T$. In what follows, we assume that all cylinders introduced below are contained in $\Omega_T$.

For $k >0$, $\sigma \in (0, 1)$ and $0< r< R$, $0< \lambda < \Lambda$, set
\[
\begin{aligned}
r_\sigma&:=(1-\sigma) r +\sigma R,
&\qquad
\bar r_\sigma&:=\frac{r_\sigma+r_{\sigma/2}}{2},\\
\eta_\sigma&:=(1-\sigma) \lambda+ \sigma \Lambda,
&
\bar\eta_\sigma&:=\frac{\eta_\sigma+\eta_{\sigma/2}}{2},
\end{aligned}
\]
and define the time-dependent levels
$$
k_\sigma(t):= k(1-\sigma) -\tilde{C}\,\mathcal{K}(t),\quad\quad\quad \bar{k}_\sigma(t):=\frac{k_\sigma(t)+k_{\sigma/2}(t)}{2},
$$
where the constant $\tilde{C} >0$ will be chosen later and
$$\mathcal{K}(t):= k^{1-q}
\int\limits^{t}_{t_0-\Lambda}\,\,\int\limits_{\mathbb{R}^N\setminus B_R(x_0)}\frac{[u_+(x, \tau)]^{p-1}}{|x-x_0|^{N+sp}}dx d\tau.$$
We also set
\[
\begin{aligned}
B_\sigma&:=B_{r_\sigma}(x_0),
&\qquad
\bar B_\sigma&:=B_{\bar r_\sigma}(x_0),\\
Q_\sigma&:=B_\sigma\times(t_0-\eta_\sigma,t_0),
&
\bar Q_\sigma&:=\bar B_\sigma\times(t_0-\bar\eta_\sigma,t_0).
\end{aligned}
\]
Define the cutoff function $\zeta(x, t):=\zeta_1(x)\zeta_2(t)$, where $\zeta_1(x)\in C^1_0(\bar{B}_\sigma)$ and $\zeta_2(t)\in C^1(\mathbb{R})$ satisfy the following conditions:
$$0\leqslant \zeta_1(x)\leqslant 1, \quad \zeta_1(x)=1 \quad \text{in}\quad B_{\sigma/2},\quad \text{and}\quad |\nabla \zeta_1(x)|\leqslant\frac{2}{\sigma (R-r)},$$
$$
 0\leqslant \zeta_2(t)\leqslant 1,\quad \zeta_2(t)=
\begin{cases}
0, & t\leqslant t_0-\bar\eta_\sigma,\\
1, & t\geqslant t_0-\eta_{\sigma/2},
\end{cases} \quad
|\zeta'_2(t)|\leqslant \frac{2}{\sigma (\Lambda-\lambda)}.
$$
Furthermore, we assume that
\begin{equation}\label{eqUpper1}
k\geqslant 4 \tilde{C}\,\mathcal{K}(t_0),\quad \text{i. e.}\quad  k\geqslant \bigg(4 \tilde{C}\int\limits^{t_0}_{t_0-\Lambda}\,\,\int\limits_{\mathbb{R}^N\setminus B_R(x_0)}\frac{[u_+(x, \tau)]^{p-1}}{|x-x_0|^{N+sp}}dx d\tau\bigg)^{\frac{1}{q}}.
\end{equation} 
We  use inequality \eqref{eqenergy} for the function $u^+_{\bar{k}_\sigma(t)}=(u-\bar{k}_\sigma(t))_+$  and obtain
 \begin{equation}\label{eqUpper2}
 \begin{aligned}
     &\sup\limits_{t_0-\bar{\eta}_{\sigma}\leqslant \tau \leqslant t_0}\,\,\int\limits_{\bar{B}_\sigma\times\{\tau\}}\,\,\int\limits^u_{\bar{k}_\sigma(\tau)}
 z^{q-1}(z-\bar{k}_\sigma(\tau))\,dz\zeta^p\,dx\\&+\int\limits^{t_0}_{t_0-\bar{\eta}_\sigma}\int\limits_{\bar{B}_\sigma}
 \int\limits_{\bar{B}_\sigma}
 \frac{|u^+_{\bar{k}_\sigma(\tau)}(x, \tau) -u^+_{\bar{k}_\sigma(\tau)}(y, \tau)|^p}{|x-y|^{N+sp}}\min[\zeta(x, \tau), \zeta(y, \tau)]^p\,dx\,dy\,d\tau
 \\&\leqslant -\tilde{C}^q\,\iint\limits_{Q_\sigma}\mathcal{K}'(\tau)\,\int\limits^u_{\bar{k}_\sigma(\tau)}z^{q-1}\,
 dz\zeta^p\,dx\,d\tau+ \gamma \iint\limits_{Q_\sigma}\int\limits^u_{\bar{k}_\sigma(\tau)}
 z^{q-1}(z-\bar{k}_\sigma(\tau))\,dz\zeta^{p-1}\,|\zeta'_\tau|dx d\tau\\&+\gamma \iint\limits_{Q_\sigma}\int\limits_{B_\sigma}\big[u^+_{\bar{k}_\sigma(\tau)}(x, \tau)+u^+_{\bar{k}_\sigma(\tau)}(y, \tau)\big]^p\,\frac{|\zeta(x, \tau)-\zeta(y, \tau)|^p}{|x-y|^{N+sp}}dx dy d\tau\\&+\gamma \iint\limits_{Q_\sigma}
 u^+_{\bar{k}_\sigma(\tau)}(x, \tau)\zeta^p(x, \tau)\bigg(\int\limits_{\mathbb{R}^N\setminus B_\sigma}\frac{\big[u^+_{\bar{k}_\sigma(\tau)}(y, \tau)\big]^{p-1}}{|x-y|^{N+sp}} dy\bigg)dx d\tau=\mathcal{J}_1+\mathcal{J}_2+\mathcal{J}_3+\mathcal{J}_4.
 \end{aligned}
 \end{equation}
We have in $Q_\sigma \cap\big\{u\geqslant \bar{k}_{\sigma}(\cdot)\big\}$
$$u\leqslant \frac{\gamma}{\sigma}(u-k_\sigma(t))_+,$$
so, we estimate $\mathcal{J}_2$ and $\mathcal{J}_3$ as follows
\begin{equation*}
\mathcal{J}_2 +\mathcal{J}_3 \leqslant \frac{\gamma}{\sigma^{q+1}(\Lambda-\lambda)}\iint\limits_{Q_\sigma}[u^+_{k_\sigma(t)}]^{q+1}\,dxdt+
\frac{\gamma}{\sigma^{p}(R-r)^{sp}}\iint\limits_{Q_\sigma}[u^+_{k_\sigma(t)}]^p\,dx dt.
\end{equation*}
Choosing $k$ such that
\begin{equation}\label{eqUpper3}
k\geqslant \bigg(\frac{\Lambda-\lambda}{(R-r)^{sp}}\bigg)^{\frac{1}{q+1-p}},
\end{equation}
we obtain
\begin{equation*}
\mathcal{J}_2 +\mathcal{J}_3 \leqslant \frac{\gamma}{\sigma^{q+1}(\Lambda-\lambda)}\iint\limits_{Q_\sigma}[u^+_{k_\sigma(t)}]^{q+1}\,dx dt.
\end{equation*}
Let $\bar{\epsilon}\in(0,1)$ be specified later.  In the case $0< q< 1$, assuming additionally that 
\begin{equation}\label{eqUpper4}
k\geqslant \bar{\epsilon} M_1,\quad M_\sigma:=\sup\limits_{Q_\sigma}\,u,
\end{equation}
and applying algebraic Lemma \ref{lemAlgebraic} we obtain
\begin{equation}\label{eqUpper5}
\begin{cases}
\int\limits^u_{\bar{k}_\sigma(\tau)}z^{q-1}\,dz \geqslant \gamma\,u^+_{\bar{k}_\sigma(\tau)}\Big[u+\bar{k}_\sigma(\tau)\Big]^{q-1}\geqslant \gamma u^+_{\bar{k}_\sigma(\tau)}\,k^{q-1},\quad \text{if}\quad q\geqslant 1,\\
\int\limits^u_{\bar{k}_\sigma(\tau)}z^{q-1}\,dz \geqslant \gamma u^+_{\bar{k}_\sigma(\tau)}\Big[u+\bar{k}_\sigma(\tau)\Big]^{q-1}\geqslant \gamma u^+_{\bar{k}_\sigma(\tau)}\,\Big[M_\sigma+k\Big]^{q-1}\geqslant \gamma \bar{\epsilon}\,u^+_{\bar{k}_\sigma(\tau)}\,k^{q-1},\,\, \text{if}\,\, 0< q< 1,
\end{cases}
\end{equation}
and hence,
\begin{equation*}
\begin{aligned}
    \mathcal{J}_1&\leqslant -\tilde{C}\bar{\epsilon} \frac{k^{q-1}}{\gamma}\iint\limits_{Q_\sigma}u^+_{\bar{k}_\sigma(\tau)}\,\mathcal{K}'(\tau)\zeta^p\,dx d\tau
\\&=-\frac{\tilde{C}\bar{\epsilon}}{\gamma}\iint\limits_{Q_\sigma}u^+_{\bar{k}_\sigma(\tau)}(x, \tau)\,\zeta^p(x, \tau)\bigg(\int\limits_{\mathbb{R}^N\setminus B_R(x_0)}\frac{[u_+(y, \tau)]^{p-1}}{|y-x_0|^{N+sp}}\,dy\bigg)\,dxd\tau.
\end{aligned}
\end{equation*}
Finally, using the facts that
$$|y-x|\geqslant \frac{R-\bar{r}_\sigma}{R} |y-x_0|\geqslant \frac{(1-\frac{3}{4}\sigma)(R-r)}{R}|y-x_0|\geqslant \frac{1}{4}\frac{R-r}{R}|y-x_0|,\,\, x\in \bar{B}_\sigma,\,\, y\in \mathbb{R}^N\setminus B_R(x_0),$$
and
$$|y-x|\geqslant r_\sigma-\bar{r}_\sigma= \frac{1}{4}\sigma\,(R-r),\quad x\in \bar{B}_\sigma,\quad y\in \mathbb{R}^N\setminus B_\sigma,$$
we obtain
\begin{equation*}
\begin{aligned}
\mathcal{J}_4&\leqslant\gamma \iint\limits_{Q_\sigma}
 u^+_{\bar{k}_\sigma(\tau)}(x, \tau)\zeta^p(x, \tau)\bigg(\sup\limits_{x\in \bar{B}_\sigma}\int\limits_{\mathbb{R}^N\setminus B_R(x_0)}\frac{\big[u^+_{\bar{k}_\sigma(\tau)}(y, \tau)\big]^{p-1}}{|x-y|^{N+sp}} dy\bigg)dx d\tau \\&+\gamma \iint\limits_{Q_\sigma}
 u^+_{\bar{k}_\sigma(\tau)}(x, \tau)\zeta^p(x, \tau)\bigg(\sup\limits_{x\in \bar{B}_\sigma}\int\limits_{B_R(x_0)\setminus B_\sigma}\frac{\big[u^+_{\bar{k}_\sigma(\tau)}(y, \tau)\big]^{p-1}}{|x-y|^{N+sp}} dy\bigg)dx d\tau \\&\leqslant
 \gamma \Big(\frac{R}{R-r}\Big)^{N+sp}\iint\limits_{Q_\sigma}
 u^+_{\bar{k}_\sigma(\tau)}(x, \tau)\zeta^p(x, \tau)\bigg(\frac{k^{p-1}}{R^{sp}}+\int\limits_{\mathbb{R}^N\setminus B_R(x_0)}\frac{\big[u_+(y, \tau)\big]^{p-1}}{|y-x_0|^{N+sp}} dy\bigg)dx d\tau \\&+\frac{\gamma M^{p-1}_1 R^N}{(\sigma(R-r))^{N+sp}}\iint\limits_{Q_\sigma}
 u^+_{\bar{k}_\sigma(\tau)}(x, \tau)dx d\tau \\&\leqslant \gamma \Big(\frac{R}{R-r}\Big)^{N+sp}\iint\limits_{Q_\sigma}
 u^+_{\bar{k}_\sigma(\tau)}(x, \tau)\zeta^p(x, \tau)\bigg(\int\limits_{\mathbb{R}^N\setminus B_R(x_0)}\frac{\big[u_+(y, \tau)\big]^{p-1}}{|y-x_0|^{N+sp}} dy\bigg)dx d\tau \\&+ \gamma \frac{\bar{\epsilon}^{1-p}}{\sigma^{N+sp}} \Big(\frac{R}{R-r}\Big)^{N+sp} (R-r)^{-sp} k^{p-1}\iint\limits_{Q_\sigma}
 u^+_{\bar{k}_\sigma(\tau)}(x, \tau)dx d\tau \\&\leqslant \gamma \Big(\frac{R}{R-r}\Big)^{N+sp}\iint\limits_{Q_\sigma}
 u^+_{\bar{k}_\sigma(\tau)}(x, \tau)\zeta^p(x, \tau)\bigg(\int\limits_{\mathbb{R}^N\setminus B_R(x_0)}\frac{\big[u_+(y, \tau)\big]^{p-1}}{|y-x_0|^{N+sp}} dy\bigg)dx d\tau \\&+ \gamma \frac{\bar{\epsilon}^{1-p}(\Lambda-\lambda)^{-1}}{\sigma^{N+sp+q}} \Big(\frac{R}{R-r}\Big)^{N+sp} \iint\limits_{Q_\sigma}
 [u^+_{k_\sigma(\tau)}]^{q+1}dx d\tau.
\end{aligned}
\end{equation*}
The first term on the right-hand side of this inequality will cancel with the   right-hand side of the estimate for $\mathcal{J}_1$, 
provided that $\tilde{C}$ is chosen to satisfy $\frac{\tilde{C} \bar{\epsilon}}{\gamma} = \gamma \big(\frac{R}{R-r}\big)^{N+sp}$.

Next we will estimate the left-hand side of \eqref{eqUpper2}. Using \eqref{eqUpper5}
\begin{equation*}
\begin{aligned}
    \int\limits^{u}_{\bar{k}_\sigma(\tau)}z^{q-1}(z-\bar{k}_\sigma(\tau)) dz&\geqslant \int\limits^{u}_{\frac{1}{2}(\bar{k}_\sigma(\tau)+u)}z^{q-1}(z-\bar{k}_\sigma(\tau))\,dz
\\&\geqslant \frac{1}{2}(u-\bar{k}_\sigma(\tau))\,\int\limits^{u}_{\frac{1}{2}(\bar{k}_\sigma(\tau)+u)}z^{q-1}\,dz
\geqslant \frac{\bar{\epsilon}}{\gamma}k^{q-1}\,\big[u^+_{\bar{k}_\sigma(\tau)}\big]^2.
\end{aligned}
\end{equation*}
Therefore, collecting the previous estimates, inequality \eqref{eqUpper2} can be rewritten as
\begin{equation}\label{eqUpper6}
\begin{aligned}
&k^{q-1}\,\sup\limits_{t_0-\bar{\eta}_{\sigma}\leqslant \tau \leqslant t_0}\int\limits_{\bar{B}_\sigma\times\{\tau\}}\big[u^+_{\bar{k}_\sigma(\tau)}\big]^2\,\zeta^p\,dx \\&+
\int\limits^{t_0}_{t_0-\bar{\eta}_\sigma}\,\,\int\limits_{\bar{B}_\sigma}\int\limits_{\bar{B}_\sigma}
 \frac{|u^+_{\bar{k}_\sigma(\tau)}(x, \tau) -u^+_{\bar{k}_\sigma(\tau)}(y, \tau)|^p}{|x-y|^{N+sp}}\min[\zeta(x, \tau), \zeta(y, \tau)]^p\,dx\,dy\,d\tau
 \\&\leqslant \frac{\gamma}{\sigma^\gamma \bar{\epsilon}^\gamma}\,(\Lambda-\lambda)^{-1} \Big(\frac{R}{R-r}\Big)^{N+sp} \iint\limits_{Q_\sigma}
 [u^+_{k_\sigma(\tau)}]^{q+1} dx d\tau.
 \end{aligned}
\end{equation}
Now, using the Sobolev embedding theorem and H\"older's inequality, we obtain 
\begin{equation*}
    \begin{aligned}
\iint\limits_{Q_{\sigma/2}}[u^+_{k_{\sigma/2}(\tau)}]^{q+1} &dx d\tau\leqslant
\Big(\iint\limits_{\bar{Q}_\sigma}\big[u^+_{\bar{k}_\sigma(\tau)}(x, \tau)\big]^{p(1+\frac{2s}{N})}[\zeta(x, \tau)]^{p(1+\frac{2 s}{N})}\,dx d\tau\Big)^{\frac{(q+1) N}{p(N+2s)}} \\&\times \big|Q_\sigma \cap\big\{u\geqslant k_{\sigma/2}(\cdot)\big\}\big|^{1-\frac{(q+1) N}{p(N+2s)}}\leqslant  \gamma \bigg(\sup\limits_{t_0-\bar{\eta}_{\sigma}\leqslant \tau \leqslant t_0}\,\,\int\limits_{\bar{B}_\sigma\times\{\tau\}}\big[u^+_{\bar{k}_\sigma(\tau)}\big]^2\,\zeta^p\,dx\bigg)^{\frac{s(q+1)}{N+2s}}\\&\times
\bigg(\int\limits^{t_0}_{t_0-\bar{\eta}_\sigma}\,\,\int\limits_{\bar{B}_\sigma}\int\limits_{\bar{B}_\sigma}
 \frac{|u^+_{\bar{k}_\sigma(\tau)}(x, \tau)\zeta(x, \tau) -u^+_{\bar{k}_\sigma(\tau)}(y, \tau)\zeta(y,\tau)|^p}{|x-y|^{N+sp}}\,dx\,dy\,d\tau\\&
+\frac{\gamma}{\sigma^\gamma(R-r)^{sp}}\iint\limits_{\bar{Q}_\sigma}[u^+_{\bar{k}_\sigma(\tau)}]^p dx\,d\tau \bigg)^{\frac{(q+1)N}{p(N+2s)}} \big|Q_\sigma \cap\big\{u\geqslant k_{\sigma/2}(\cdot)\big\}\big|^{1-\frac{(q+1)N}{p(N+2s)}},
\end{aligned}
\end{equation*}
evaluating the integrand of the first term by the inequality
\begin{multline*}
|u^+_{\bar{k}_\sigma(\tau)}(x, \tau)\zeta(x, \tau) -u^+_{\bar{k}_\sigma(\tau)}(y, \tau)\zeta(y,\tau)|^p \\\leqslant
\gamma|u^+_{\bar{k}_\sigma(\tau)}(x, \tau) -u^+_{\bar{k}_\sigma(\tau)}(y, \tau)|^p\zeta^p(x, \tau)+\gamma |\zeta(x, \tau) -\zeta(y,\tau)|^p[u^+_{\bar{k}_\sigma(\tau)}(x, \tau)]^p,
\end{multline*}
and using \eqref{eqUpper6} we arrive at
\begin{equation*}
\begin{aligned}
    \mathbf{Y}(\sigma/2)&:=\iint\limits_{Q_{\sigma/2}}[u^+_{k_{\sigma/2}(\tau)}]^{q+1}dx d\tau\leqslant \frac{\gamma}{\bar{\epsilon}^\gamma\sigma^\gamma}
k^{(1-q)\frac{s(q+1)}{N+2s}-(q+1)(1-\frac{(q+1)N}{p(N+2s)})}\\&\times (\Lambda-\lambda)^{-\frac{(q+1)s}{N+2s}-\frac{(q+1)N}{p(N+2s)}} \Big(\frac{R}{R-r}\Big)^{(N+sp)(\frac{(q+1)s}{N+2s}+\frac{(q+1)N}{p(N+2s)})}\big[\mathbf{Y}(\sigma)\big]^{1+\frac{s(q+1)}{N+2s}}.
\end{aligned}
\end{equation*}
By interpolation this inequality yields $\lim\limits_{\sigma \rightarrow 0}\mathbf{Y}(\sigma)=0$, provided that
$$\mathbf{Y}(1)\leqslant 
\frac{\bar{\epsilon}^\gamma}{\gamma}\,k^{\frac{\varkappa+sp}{sp}}(\Lambda-\lambda)^{\frac{N+sp}{sp}}\Big(\frac{R-r}{R}\Big)^{\frac{(N+sp)^2}{sp}},
\quad \varkappa=pq+N(p-1-q),$$
which, together with our choices \eqref{eqUpper1}, \eqref{eqUpper3} and \eqref{eqUpper4}, yields
\begin{equation*}
\begin{aligned}
    M_{0}&\leqslant \bar{\epsilon}\,M_1+\frac{\gamma}{ \bar{\epsilon}^\gamma}\bigg((\Lambda-\lambda)^{-\frac{N+sp}{sp}}\Big(\frac{R}{R-r}\Big)^{\frac{(N+sp)^2}{sp}}\iint\limits_{Q_1} u^{q+1}\,dx\,d\tau\bigg)^{\frac{sp}{\varkappa+sp}}+
\gamma \Big(\frac{\Lambda-\lambda}{(R-r)^{sp}}\Big)^{\frac{1}{q+1-p}} \\&+\frac{\gamma}{\bar{\epsilon}^\gamma} \bigg( \Big(\frac{R}{R-r}\Big)^{N+sp}\int\limits^{t_0}_{t_0-\Lambda}\,\,\int\limits_{\mathbb{R}^N\setminus B_R(x_0)}\frac{[u_+(x, \tau)]^{p-1}}{|x-x_0|^{N+sp}}dx d\tau\bigg)^{\frac1q}\\&\leqslant \bar{\epsilon}\,M_1+\frac{\gamma}{ \bar{\epsilon}^\gamma}M^{\frac{(q+2-p)sp}{\varkappa+sp}}_1\bigg((\Lambda-\lambda)^{-\frac{N+sp}{sp}}\Big(\frac{R}{R-r}\Big)^{\frac{(N+sp)^2}{sp}}\iint\limits_{Q_1} u^{p-1}\,dx\,d\tau\bigg)^{\frac{sp}{\varkappa+sp}}\\&+
\gamma \Big(\frac{\Lambda-\lambda}{(R-r)^{sp}}\Big)^{\frac{1}{q+1-p}} +\frac{\gamma}{\bar{\epsilon}^\gamma} \bigg( \Big(\frac{R}{R-r}\Big)^{N+sp}\int\limits^{t_0}_{t_0-\Lambda}\,\,\int\limits_{\mathbb{R}^N\setminus B_R(x_0)}\frac{[u_+(x, \tau)]^{p-1}}{|x-x_0|^{N+sp}}dx d\tau\bigg)^{\frac1q}.
\end{aligned}
\end{equation*}
Using Young's inequality and the fact that $\frac{(q+2-p)sp}{\varkappa+sp} <1$, it follows that
\begin{equation}\label{eqUpper7}
\begin{aligned}
    M_{0}&\leqslant 2\bar{\epsilon}\,M_1+\frac{\gamma}{ \bar{\epsilon}^\gamma}\bigg((\Lambda-\lambda)^{-\frac{N+sp}{sp}}\Big(\frac{R}{R-r}\Big)^{\frac{(N+sp)^2}{sp}}\iint\limits_{Q_1} u^{p-1}\,dx\,d\tau\bigg)^{\frac{sp}{\widetilde \varkappa}}\\&+
\gamma \Big(\frac{\Lambda-\lambda}{(R-r)^{sp}}\Big)^{\frac{1}{q+1-p}}+\frac{\gamma}{\bar{\epsilon}^\gamma} \bigg( \Big(\frac{R}{R-r}\Big)^{N+sp}\int\limits^{t_0}_{t_0-\Lambda}\,\,\int\limits_{\mathbb{R}^N\setminus B_R(x_0)}\frac{[u_+(x, \tau)]^{p-1}}{|x- x_0|^{N+sp}}dx d\tau\bigg)^{\frac{1}{q}}.
\end{aligned}
\end{equation}
We recall that  $\widetilde \varkappa=sp(p-1)+N(p-1-q) > 0$.

Next, we use \eqref{eqUpper7} to apply an interpolation argument. 
For fixed $\rho$, $\theta>0$ and $\sigma \in (0, 1)$, choose $r$ and $\lambda$ such that $\frac{1}{2}\rho \leqslant r < r(1+\sigma) \leqslant \rho$ and $\frac{1}{2} \theta \leqslant \lambda < \lambda (1+\sigma)\leqslant \theta$, and apply \eqref{eqUpper7} with $R=r(1+\sigma)$ and $\Lambda=\lambda(1+\sigma)$.
\begin{equation*}
\begin{aligned}
    \sup\limits_{Q^-_{r, \lambda}(x_0, t_0)} u&\leqslant 2 \bar{\epsilon}\,\sup\limits_{Q^-_{r(1+\sigma), \lambda(1+\sigma)}(x_0, t_0)} u+\frac{\gamma}{ \bar{\epsilon}^\gamma \sigma^\gamma}\bigg[\Big(\frac{\rho^{sp}}{\theta}\Big)^{\frac{N+sp}{\widetilde \varkappa}}\bigg(\rho^{-N-sp}\iint\limits_{Q^-_{\rho, \theta}(x_0, t_0)} u^{p-1}\,dx\,d\tau\bigg)^{\frac{sp}{\widetilde \varkappa}}\\&+
 \Big(\frac{\theta}{\rho^{sp}}\Big)^{\frac{1}{q+1-p}}+ \bigg( \int\limits^{t_0}_{t_0-\theta}\,\,\int\limits_{\mathbb{R}^N\setminus B_{\frac{1}{2}\rho}(x_0)}\frac{[u_+(x, \tau)]^{p-1}}{|x-x_0|^{N+sp}}dx d\tau\bigg)^{\frac{1}{q}}\bigg].
\end{aligned}
\end{equation*}
Interpolating, we obtain
\begin{equation*}
\begin{aligned}
   \sup\limits_{Q^-_{\frac{1}{2}\rho, \frac{1}{2}\theta}(x_0, t_0)} u&\leqslant \gamma \bigg[\Big(\frac{\rho^{sp}}{\theta}\Big)^{\frac{N+sp}{\widetilde \varkappa}}\bigg(\rho^{-N-sp}\iint\limits_{Q^-_{\rho, \theta}(x_0, t_0)} u^{p-1}\,dx\,d\tau\bigg)^{\frac{sp}{\widetilde \varkappa}}+
 \Big(\frac{\theta}{\rho^{sp}}\Big)^{\frac{1}{q+1-p}}\\&+ \bigg( \int\limits^{t_0}_{t_0-\theta}\,\,\int\limits_{\mathbb{R}^N\setminus B_{\frac{1}{2}\rho}(x_0)}\frac{[u_+(x, \tau)]^{p-1}}{|x-x_0|^{N+sp}}dx d\tau\bigg)^{\frac{1}{q}}\bigg]\\&\leqslant \gamma \bigg[\Big(\frac{\rho^{sp}}{\theta}\Big)^{\frac{N+sp}{\widetilde \varkappa}}\bigg(\int\limits^{t_0}_{t_0-\theta}\,\,\int\limits_{\mathbb{R}^N} \frac{[u_+(x, \tau)]^{p-1}}{(\rho+|x-x_0|)^{N+sp}}\,dx\,d\tau\bigg)^{\frac{sp}{\widetilde \varkappa}}+
 \Big(\frac{\theta}{\rho^{sp}}\Big)^{\frac{1}{q+1-p}}\\&+ \bigg( \int\limits^{t_0}_{t_0-\theta}\,\,\int\limits_{\mathbb{R}^N}\frac{[u_+(x, \tau)]^{p-1}}{(\rho+|x-x_0|)^{N+sp}}dx d\tau\bigg)^{\frac{1}{q}}\bigg]. 
\end{aligned}
\end{equation*}
This completes the proof of Proposition \ref{pr1.2}.

\section{Proof of Theorem \ref{th1.1}}\label{sec:Harnack}

We first establish the right-hand side of  \eqref{eqHarnack2} in Theorem \ref{th1.1}.  Apply Proposition \ref{pr1.2} in the cylinder
$$Q^-_{\rho, \delta \theta}(x_0, t_0)=B_\rho(x_0)\times(t_0-\delta \theta, t_0),\qquad \theta= \Big(\frac{1}{4\bar{C}_3}u_0\Big)^{q-p+1}\rho^{sp},\quad u_0:=u(x_0, t_0)$$
with $\delta \in (0, 1)$  and $\bar{C}_3 > 0$ is the constant from Proposition \ref{pr1.2}. We obtain
\begin{equation*}
\begin{aligned}
    u_0\leqslant \sup\limits_{Q^-_{\frac{1}{2}\rho, \frac{1}{2}\delta \theta}(x_0, t_0)} u&\leqslant \frac{\bar{C}_3(4\bar{C}_3)^{\frac{(N+sp)(q-p+1)}{\widetilde \varkappa}}}{\delta^{\frac{N+sp}{\widetilde \varkappa}}u^{\frac{(N+sp)(q-p+1)}{\widetilde \varkappa}}_0}  \bigg( \int\limits^{t_0}_{t_0-\delta \theta}\,\,\int\limits_{\mathbb{R}^N}\frac{[u_+(x, \tau)]^{p-1}}{(\rho+|x-x_0|)^{N+sp}}dx d\tau\bigg)^{\frac{sp}{\widetilde \varkappa}}\\&+
\bar{C}_3 \bigg( \int\limits^{t_0}_{t_0-\delta \theta}\,\,\int\limits_{\mathbb{R}^N}\frac{[u_+(x, \tau)]^{p-1}}{(\rho+|x-x_0|)^{N+sp}}dx d\tau\bigg)^{\frac1q}+\frac{1}{4}\,u_0, 
\end{aligned}
\end{equation*}
which, by a straightforward calculation, yields 
\begin{equation}\label{eqH1}
u_0\leqslant \frac{4\bar{C}_3}{\delta^{\frac{N+sp}{spq}}}\bigg(\int\limits^{t_0}_{t_0-\delta \theta}\,\,\int\limits_{\mathbb{R}^N}\frac{[u_+(x, \tau)]^{p-1}}{(\rho+|x-x_0|)^{N+sp}}dx d\tau\bigg)^{\frac{1}{q}}.
\end{equation}
This proves the right-hand side of \eqref{eqHarnack2}. 

To prove the right-hand side of \eqref{eqHarnack1},  we apply Proposition \ref{pr1.1}. In view of \eqref{eqH1}, we set
$$
\mathcal I:=\frac{1}{4\bar{C}_3}\,u_0.
$$
Then
$$ \mathcal{I}\leqslant \frac{1}{\delta^{\frac{N+sp}{spq}}}\bigg( \int\limits^{t_0}_{t_0-\delta \theta}\,\,\int\limits_{\mathbb{R}^N}\frac{[u_+(x, \tau)]^{p-1}}{(\rho+|x-x_0|)^{N+sp}}dx d\tau\bigg)^{\frac{1}{q}},\quad \theta=\mathcal{I}^{q-p+1} \rho^{sp}.$$
Choosing $\delta$ small enough and applying  Proposition \ref{pr1.1}, we arrive at the required  inequality  \eqref{eqHarnack1}, namely
\begin{equation}\label{eqH2}
\frac{1}{4 \bar{C}_3}\,u_0=\mathcal{I}\leqslant \bar{C}_2\,\inf\limits_{B_\rho(x_0)} u(\cdot, t),
\end{equation}
for all $t$ in the range
$$t_0+\frac{1}{2}\delta \Big(\frac{1}{4\bar{C}_3}u_0\Big)^{q-p+1}\rho^{sp}\leqslant t\leqslant  t_0+\delta \Big(\frac{1}{4\bar{C}_3}u_0\Big)^{q-p+1}\rho^{sp},$$
provided that
\begin{equation*}
\left(\,\,
\int\limits_{t_0-\rho^{sp}\,u_0^{q-p+1} }^{t_0+ \rho^{sp}\,u_0^{q-p+1}}\,\,\,
\int\limits_{\mathbb{R}^N\setminus B_\rho(x_0)}
\frac{[u_-(x,t)]^{p-1}}
{|x-x_0|^{N+sp}}
,dx,dt
\right)^{\frac1q}
\leqslant
\frac{1}{2 \bar{C}_1 \bar{C}_3} u_0.
\end{equation*}
The proof of the right-hand side of \eqref{eqHarnack1} is concluded by  redefining  the  constant $\delta$.

To prove the left-hand inequality in \eqref{eqHarnack1}, we use the results of \cite[Chapter 6, Section 9]{DiBGiaVes2012}, which require only the continuity of the solution and the right-hand inequality in \eqref{eqHarnack1}, established above. We leave the details of the proof to the reader.

To establish the left-hand inequality in \eqref{eqHarnack2}, we apply inequality \eqref{eqShr8} from the shrinking Lemma \ref{lemShrinking1}  with $\frac{1}{2} \epsilon k=C_2 u_0$, in the cylinder $Q^*:=B_\rho(x_0)\times (t_0-\delta \eta, t_0-\frac{1}{2} \delta \eta)$, where
$\eta:=u_0^{q-p+1}\rho^{sp}$  and $\delta \in (0, 1)$, $C_2> 0$ are the constants defined in \eqref{eqHarnack1} of Theorem \ref{th1.1}. We obtain
\begin{equation*}
\inf\limits_{t\in (t_0-\delta\eta, t_0-\frac{1}{2}\delta\eta)} \big|B_{\rho}(x_0)\cap\big\{u(\cdot, t)\leqslant C_2 u_0\big\}\big| \leqslant\frac{\gamma(C_2)\, u_0^q}{\int\limits^{t_0-\frac{1}{2}\delta \eta}_{t_0-\delta\eta}\int\limits_{\mathbb{R}^N}\frac{[u_+(y, \tau)]^{p-1}}{(\rho+|x_0-y|)^{N+sp}}dy d\tau}|B_\rho(x_0)|.
\end{equation*}
By the left-hand inequality in \eqref{eqHarnack1} we have
$$
\inf\limits_{t\in (t_0-\delta\eta, t_0-\frac{1}{2}\delta \eta)} \big|B_{\rho}(x_0)\cap\big\{u(\cdot, t)\leqslant C_2 u_0\big\}\big|=|B_\rho(x_0)|.
$$
Combining this with the preceding estimate, we obtain
$$
\int\limits^{t_0-\frac{1}{2}\delta \eta}_{t_0-\delta\eta}\,\int\limits_{\mathbb{R}^N}\frac{[u_+(y, \tau)]^{p-1}}{(\rho+|x_0-y|)^{N+sp}}dy d\tau\leqslant \gamma u_0^q.$$
This completes the proof of Theorem \ref{th1.1}.

\section*{Acknowledgements} 
Igor Skrypnik and Mariia Savchenko were supported by the National
Research Foundation of Ukraine, project No. 2025.07/0369 “Qualitative methods of nonlinear
analysis of heterogeneous structures.”

\bibliographystyle{abbrvurl}
\bibliography{ref}

\end{document}